\documentclass[psamsfonts,11pt]{amsart}
\usepackage{amsmath,amssymb,amsthm,amscd} 
\usepackage{amssymb,fge} 
\usepackage{tikz-cd}
\usepackage{mathtools}
\usepackage{comment}
\usepackage{graphicx}

\newtheorem{thm}{Theorem}
\newtheorem{cor}[thm]{Corollary}
\newtheorem{prop}[thm]{Proposition}
\newtheorem{lem}[thm]{Lemma}

\newtheorem{defn}[thm]{Definition}

\newcommand\scalemath[2]{\scalebox{#1}{\mbox{\ensuremath{\displaystyle #2}}}}

\theoremstyle{definition} 
\newtheorem{defin}[thm]{Definition}
\newtheorem{rem}[thm]{Remark}

\theoremstyle{remark}

\numberwithin{thm}{section}
\numberwithin{equation}{thm}
\newcommand{\be}{\begin{equation}}
\newcommand{\ee}{\end{equation}}

\renewcommand{\P}{\ensuremath{{\mathbb{P}}}}
\newcommand{\Q}{\ensuremath{{\mathbb{Q}}}}

\DeclareMathOperator{\Prep}{Prep}

\DeclareMathOperator{\Orb}{Orb}
\DeclareMathOperator{\Rat}{Rat}
\usepackage[utf8]{inputenc}
\title{Counting points by height in semigroup orbits}
\author{Jason P. Bell, Wade Hindes, Xiao Zhong}
\date{April 2023}

\begin{document}
\maketitle
\begin{abstract} We improve known estimates for the number of points of bounded height in semigroup orbits of polarized dynamical systems. In particular, we give exact  asymptotics for generic semigroups acting on the projective line. The main new ingredient is the Wiener-Ikehara Tauberian theorem, which we use to count functions in semigroups of bounded degree.

\end{abstract} 
\section{Introduction}
A general principle in arithmetic dynamics is that the arithmetic behavior of orbits of rational points of algebraic varieties under self-maps can often be understood in terms of dynamical properties of the associated dynamical systems.  In particular, dynamical properties of algebraic dynamical systems often put strong constraints on the shapes of orbits of rational points.  It is then natural to study the extent to which measures of complexity of orbits of rational points are reflected by the corresponding measures for complexity of dynamical systems.  Typically, one uses some notion of asymptotic behavior of heights in orbits as a measure of the complexity of an orbit and uses some notion of degree for iterates to measure the complexity of a dynamical system.  The Kawaguchi-Silverman conjecture \cite{KS}, which asserts that the dynamical degree should determine the height growth for Zariski dense orbits of $\overline{\mathbb{Q}}$-points, is perhaps the most striking prediction concerning the connection between these seemingly unrelated notions of complexity.  Their conjecture has since been verified in a number of cases \cite{KSConj1,KSConj2,KSConj3,KSConj4,KSConj5}, but remains open.

In this paper, we build on earlier work of the second-named author \cite{Wade:MathZ}, and investigate the growth of heights under semigroup orbits. More precisely, we consider the setting when one has a projective variety $V$ and a finite set $S=\{\phi_1,\ldots ,\phi_r\}$ of regular self-maps defined over a number field $K$ that are polarizable with respect to a common ample invertible sheaf; \emph{i.e.}, there is an ample invertible sheaf $\mathcal{L}$ and integers $d_1,\ldots ,d_r\ge 2$ such that $\phi_i^*(\mathcal{L}) \cong \mathcal{L}^{\otimes d_i}$ for $i=1,\ldots ,r$.  We note that this polarizable condition holds for every set of regular maps that are not automorphisms when $V=\mathbb{P}^N$ with $\mathcal{L}=\mathcal{O}(1)$. Now let $M_S$ denote the semigroup of endomorphisms of $V$ generated by $S$ under composition. In particular, when $S$ is polarized one has a natural notion of degree for each $f\in M_S$ by declaring that the degree of $f$, denoted by ${\rm deg}_{\mathcal{L}}(f)$, is the unique integer $d$ such that $f^*\mathcal{L}\cong\mathcal{L}^{\otimes d}$. Likewise, we may choose a height function $h_{\mathcal{L}}:V\rightarrow\mathbb{R}$ associated to $\mathcal{L}$, and to any point $P\in V(\overline{\mathbb{Q}})$ study the \emph{orbit growth} of $P$, which is the function
$$\#\{f\in M_S \colon h_{\mathcal{L}}(f(P))\le X\}.$$  Then the general principle stated above suggests that the asymptotics of the orbit growth should be heavily constrained by the asymptotic behavior of the corresponding \emph{degree growth} function
$$\#\{f\in M_S \colon {\rm deg}_{\mathcal{L}}(f)\le X\}.$$
It is relatively straightforward to give upper bounds for the orbit growth in terms of the degree growth.  Indeed, Kawaguchi and Silverman \cite{KS} showed that the analogous phenomenon holds for dominant rational self-maps. On the other hand, it is very difficult in general to give strong lower bounds for the orbit growth in terms of the degree growth.  The reason for this is that distinct maps in $M_S$ can send a point $P$ to a common point $Q$ and this sort of ``collapse'' can, in some cases, be extreme.  For example, if $P$ is a common fixed point of all elements of $S$ then the orbit growth function is $O(1)$, while the degree growth function can be exponential in $X$.  

In an earlier work, the second-named author \cite{Wade:MathZ} considered orbit growth and degree growth for semigroups of polarizable morphisms.  In this paper, we extend those results by getting optimal asymptotics for degree growth and getting optimal asymptotics for orbit growth for certain classes of semigroups of morphisms of the projective line. However, similar to the case of iterating a single function, we must avoid preperiodic points, defined below. 
\begin{defn} Let $S=\{ \phi_1,\ldots ,\phi_r\}$ be a set of endomorphisms on a variety $V$.  We say that a point $P\in V$ is \emph{preperiodic} for $M_S$ if there exist $f,g\in M_S$, with $g$ not the identity, such that $g\circ f(P)=f(P)$ 
\end{defn}   
Our first main result is the following. In what follows, given real functions $F$ and $G$, we write $F\ll G$ if there is a constant $C$ such that $F(x)\leq C\cdot G(x)$ holds for all $x$ sufficiently large. Likewise, we write $F\asymp G$ to mean both $F\ll G$ and $G\ll F$. Finally, we write $F\sim G$ to mean $\displaystyle{\lim_{x\rightarrow\infty} F(x)/G(x)=1}$.   
\begin{thm}{\label{thm:functioncount}} Let $V$ be a projective variety, let $S:=\{\phi_1,\dots,\phi_r\}$ be a polarized set of endomorphisms on $V$, and say $\phi_i^*\mathcal{L}\cong\mathcal{L}^{\otimes d_i}$ for some $d_i>1$. Moreover, assume that $r\geq2$ and let $\rho$ be the unique real number satisfying
\[\frac{1}{d_1^\rho}+\dots+\frac{1}{d_r^\rho}=1.\]
Then the following statements hold for all non-preperiodic $P\in V$:\vspace{.25cm}
\begin{enumerate}
\item[\textup{(1)}] In all cases, $\#\{f\in M_S\,:\, h_{\mathcal{L}}(f(P))\leq X\}\ll X^{\rho}$. \vspace{.3cm} 
\item[\textup{(2)}] If $M_S$ contains a free, finitely generated sub-semigroup, then there is a positive constant $\rho'\leq \rho$ such that  \vspace{.15cm} 
\[X^{\rho'}\ll \;\#\{f\in M_S\,:\, h_{\mathcal{L}}(f(P))\leq X\}\ll X^{\rho}.\vspace{.3cm}\]
\item[\textup{(3)}] If $M_S$ is free, then $\#\{f\in M_S\,:\, h_{\mathcal{L}}(f(P))\leq X\}\asymp X^{\rho}$. \vspace{.3cm} 
\item[\textup{(4)}] If $M_S$ is free and the $d_1,\ldots ,d_r$ are not all pairwise multiplicatively dependent, then there is a positive constant $c$ such that \vspace{.1cm} 
\[\#\{f\in M_S\,:\, h_{\mathcal{L}}(f(P))\leq X\}\sim cX^{\rho}.\vspace{.25cm}\]
\end{enumerate}
Here in all cases, the implicit constants depend on $P$ and $S$. 
\end{thm}
\begin{rem} In particular, if the individual sets of preperiodic points $\Prep(\phi)$ of each $\phi\in S$ are not all identical, then we obtain bounds of the form 
\[X^{\rho'}\ll \;\#\{f\in M_S\,:\, h_{\mathcal{L}}(f(P))\leq X\}\ll X^{\rho}\] 
for all non-preperiodic $P \in V$ by \cite[Theorem 1.3]{Tits}. 
\end{rem} 
\begin{rem} In fact, an exact asymptotic as in statement (4) of Theorem \ref{thm:functioncount} is not possible when the degrees of the maps in $S$ are all powers of a single integer. For a concrete example, consider the set of polynomials $S=\{2x^8, 3x^8\}$ and the point $P=1$. Note that $\rho=1/3$ in this case. Now for $X>0$, define the function \[\Theta(X)=\#\{f\in M_S\,:\, h(f(P))\leq X\}/X^{1/3},\] 
where $h(\cdot)$ is the usual Weil height on $\mathbb{P}^1$. Then one can show that 
\[0<\liminf \Theta(X) < \limsup \Theta(X) <\infty\] 
in this case.    
\end{rem} 
On the other hand, since every point in the total \emph{semigroup orbit} of $P$, \[\Orb_S(P):=\{f(P)\,:\,f\in M_S\},\] 
of height at most $X$ is determined by at least one function $f\in M_S$ satisfying $h_{\mathcal{L}}(f(P))\leq X$, we easily obtain a general upper bound for the number of points of bounded height in non-preperiodic orbits. However, with some additional work, we can prove that such an upper bound holds for all orbits.  
\begin{thm}\label{thm:universalupperbd} Let $V$ be a projective variety, let $S:=\{\phi_1,\dots,\phi_r\}$ be a polarized set of endomorphisms on $V$. Moreover, assume that $r\geq2$ and that $\rho$ and $\mathcal{L}$ are as in Theorem \ref{thm:functioncount}. 
Then  
\[\#\{Q\in\Orb_S(P)\,:\, h_{\mathcal{L}}(Q)\leq X\}\ll X^{\rho}\]
holds for all points $P\in V$.  
\end{thm}

On the other hand, we prove an even stronger result for generic sets of maps on $V=\mathbb{P}^1$ and establish an asymptotic for the number of \emph{points} in an orbit of bounded height, instead of just an asymptotic for the number of \emph{functions} yielding a bounded height relation. In what follows, given a positive integer $d$, we let $\Rat_d$ be the space of rational functions in one variable of degree $d$; see \cite[\S4.3]{SilvDyn} for details. 
\begin{thm}\label{thm:generic} Let $d_1,\dots, d_r$ be integers all at least four. Moreover, assume that $r\geq2$ and let $\rho$ be the unique positive real number satisfying $\sum d_i^{-\rho}=1$. Then there is a proper Zariski open subset $U\subseteq\Rat_{d_1}\times\dots\times\Rat_{d_r}$ such that 
\[\#\{Q\in \Orb_S(P)\,:\, h(Q)\leq X\}\asymp X^{\rho}\] 
for all $S\in U$ and all $P$ such that $\Orb_S(P)$ is infinite. Moreover, if in addition the $d_1,\ldots ,d_r$ are not all pairwise multiplicatively dependent, then  
\[\#\{Q\in \Orb_S(P)\,:\, h(Q)\leq X\}\sim c X^{\rho}\]
for some positive constant $c$, which depends on $P$ and $S$.  
\end{thm} 

\begin{rem} The constants in Theorem \ref{thm:generic} are not explicit, since they incorporate height bounds for rational points on curves. On the other hand, the subset $U$ is explicit and consists of all critically separate and critically simple sets of maps (cf. \cite[Theorem 1.5]{Wade:MathZ}); see section \ref{sec:P^1} for details.          
\end{rem}
\begin{rem} Similar hypotheses were recently used in \cite{Joe+Wade} to give strong lower bounds for the number of points in semigroup orbits modulo primes. Moreover, these results are analogous to known results for abelian varieties \cite{Murty}. With this analogy in mind, the number of points of height at most $X$ on an abelian variety $A$ over a number field $K$ is asymptotic to $X^{\rho/2}$, where $\rho$ is the rank of the Mordell-Weil group $A(K)$; see \cite[p.54]{Serre}. Thus the present paper, as well as the papers~\cite{Wade:MathZ,Joe+Wade}, suggest that the analogy 
\[
\left(\begin{tabular}{@{}l@{}}
arithmetic of points\\
of an abelian variety\\
\end{tabular}\right)
\quad\Longleftrightarrow\quad
\left(\begin{tabular}{@{}l@{}}
arithmetic of points in orbits\\
of a dynamical system\\
\end{tabular}\right)
\]
described in~\cite{MR4007163} and~\cite[\S6.5]{MR2884382} may be more accurate when the dynamical system on the right-hand side is generated by at least two non-commuting maps, rather than using orbits coming from iteration of a single map.    
\end{rem}
\subsection{Notation}
Throughout this paper, we use the following notation.  If $S$ is a set of self-maps, we let $M_S$ denote the semigroup generated by $S$ under composition.  If $M_S$ is free on the set $S$, we let $S_n$ (respectively, $S_{\le n}$) denote the set of $n$-fold compositions of elements of $S$ (respectively, elements that are compositions of at most $n$ elements of $S$), and for $f\in M_S$ we let $|f|$ denote the unique natural number $n$ such that $f$ is an $n$-fold composition of elements of $S$. 
\section{Counting functions in polarized systems}\label{sec:degrees}
\subsection{Functions of bounded degree}\label{degreegrowth} 
Let $V$ be a projective variety, and let $S=\{\phi_1,\dots,\phi_r\}$ be a set of endomorphisms on $V$. Moreover, we assume that $S$ is \emph{polarized} with respect to some ample line bundle $\mathcal{L}$ on $V$, meaning that there are $d_1,\dots, d_r>1$ satisfying $\phi_i^*\mathcal{L}\cong\mathcal{L}^{\otimes d_i}$. In particular, we may defined the $\mathcal{L}$-degree (or simply degree when $\mathcal{L}$ is understood) of a function $f$ in the semigroup generated by $S$ as follows: 
\[\qquad\qquad\deg_{\mathcal{L}}(f):=\prod_{j=1}^n d_{i_j}\qquad \text{for $f=\phi_{i_1}\circ\dots\circ \phi_{i_n}\in M_S$}\] 
In particular, $f^*\mathcal{L}\cong\mathcal{L}^{\otimes\deg_{\mathcal{L}}(f)}$ follows from standard results in algebraic geometry. On the other hand, if we fix a height function associated to $\mathcal{L}$ on $V$ and a point $P\in V$, then the height of $f(P)$ tends to grow according to $\deg_\mathcal{L}(f)$ (a fact used by Call and Silverman to define canonical heights). Therefore, as a first step in counting points of bounded height in orbits, we count functions of bounded degree. In fact, we will prove in this section that if $M_S$ is free, then there are positive constants $\rho$, $c_1$, and $c_2$ (depending only on the initial degrees $d_1,\dots d_r$) such that \vspace{.1cm} 
\begin{equation}\label{cor:functionount} 
c_1 X^{\rho}\leq\#\{f\in M_S\,:\, \deg_{\mathcal{L}}(f)\leq X\}\leq c_2 X^{\rho}
\vspace{.1cm} 
\end{equation}  
for all $X$ sufficiently large. Here, the constant $\rho$ is the unique positive real number satisfying \[\frac{1}{d_1^\rho}+\dots+\frac{1}{d_r^\rho}=1.\] 
In fact, we can give more exact asymptotics. However, the form of these asymptotics turns
out to depend on whether or not the degrees of the maps in $S$ are all powers of a single integer, and so we put off stating these more precise results until later.  

To begin, we collect some basic facts about counting words in free semigroups of bounded, \emph{multiplicative} weight. This problem is very similar to a classical arithmetic problem of counting bounded multiplicative compositions of integers \cite{Erdos, Survey}, though not exactly the same, since we want to allow factors of equal size (precisely to allow maps of equal degree in our main dynamical theorems). In particular, due to the lack of an exact reference, we include the relevant statements and proofs here. Let us fix some notation. Let $S=\{\alpha_1,\dots,\alpha_r\}$ be an alphabet, and let $F_S$ be the set of all finite words generated by $S$. Then $F_S$ is free semigroup under concatenation. Given a vector $\pmb{d}=(d_1,\dots, d_r)$ of positive real weights, define a multiplicative weight function $|\cdot|_{\pmb{d}}: F_S\rightarrow \mathbb{R}$ given by 
\begin{equation}\label{eq:weight} 
\quad\quad\qquad |\omega|_{\pmb{d}}:=\prod_{j=1}^md_{i_j},\qquad \omega=\alpha_{i_1}\cdots\alpha_{i_m}\in F_S.
\end{equation} 
Then, given a real number $X$, we would like to obtain an asymptotic formula for the number of words of weight at most $X$,
\begin{equation}\label{wtatmostx}
\#\{\omega\in F_S\,:\, |\omega|_{\pmb{d}}\leq X\}.  
\end{equation} 
In particular, we have the following general estimate, which assumes only that $r\geq2$ and that $\min\{d_i\}>1$. 
\begin{rem} Strictly speaking, the result below is not necessary for proving our main theorems in this paper; it is superseded by the asymptotics given in Theorem \ref{thm:dirichlet}. Nevertheless, we include this result since it may be useful for future work. For one, the more precise asymptotics to come assume integral weights, and non-integral weights do appear elsewhere in dynamics \cite{Kawaguchi,SilvK3}. Hence, the result below may still be useful in a more general setting.       
\end{rem}
\begin{prop}\label{thm:wordcount} Let $S=\{\alpha_1,\dots,\alpha_r\}$ be an alphabet, let $F_S$ be the free semigroup generated by $S$, and let $\pmb{d}=(d_1,\dots, d_r)$ be a vector of positive weights. Assume that $r\geq2$, that $\min\{d_1,\dots,d_r\}>1$, and let $\rho$ be the unique real number satisfying $d_1^{-\rho}+\dots+d_r^{-\rho}=1$. Then there are positive constants $c_1$ and $c_2$ such that \vspace{.1cm} 
\[c_1\, X^{\rho}\leq \#\{\omega\in F_S\,:\, |\omega|_{\pmb{d}}\leq X\}\leq c_2\, X^{\rho} \vspace{.05cm} \]  
holds for all $X$ sufficiently large.    
\end{prop}
\begin{proof} Let $w(x)$ be the number of words in $F_S$ of weight at most $x$, where we take the weight of the empty word to be $1$. Then,
\[ w(x)=\begin{cases} 
      0 & \text{if}\;0\leq x<1, \\
      1+\displaystyle{\sum_{i=1}^rw\Big(\frac{x}{d_i}\Big)}&\text{if}\;1\leq x<\infty.
   \end{cases}
\] 
Indeed, there are no words of weight less than $1$, and any non-empty word $\omega$ of weight at most $x$ is uniquely of the form $\omega=\alpha_{i_1}\omega'$ for some word $\omega'$ of weight at most $x/d_{i_1}$ (and there are exactly $w(x/d_{i_1})$ such $\omega'$). Now let $W(x):=w(x)+\frac{1}{r-1}$ and note that it suffices to show that $W(x)\asymp x^{\rho}$, to conclude the same for $w(x)$; the reason we consider $W$ instead of $w$ is that it satisfies a slightly simpler recursion: 
\begin{equation}\label{eq:functional}
W(x)=\begin{cases} 
      1 & \text{if}\;0\leq x<1, \\
      \displaystyle{\sum_{i=1}^rW\Big(\frac{x}{d_i}\Big)}&\text{if}\;1\leq x<\infty,
   \end{cases}
\end{equation}  
which follows easily from $w$'s functional equation. Now assume that the $d$'s are arranged in non-decreasing order, $d_1\leq \dots\leq d_r$, and let $c_1=W(1)/d_r^\rho$ and $c_2=W(d_r)$. We will show that \vspace{.1cm}   
\begin{equation}\label{eq:statement}
c_1x^{\rho}\leq W(x)\leq c_2x^{\rho}\qquad \text{for all\, $1\leq x\leq d_r\cdot d_1^n$}
\end{equation} 
and all $n\geq0$ by induction on $n$. In particular, since $d_1>1$, it follows that $c_1x^{\rho}\leq W(x)\leq c_2x^{\rho}$ for all $x\geq1$. Note first that if $1\leq x\leq d_r$, then
\vspace{.05cm}
\begin{equation*}
\begin{split} 
c_1x^{\rho}=\frac{W(1)}{d_r^\rho}x^{\rho}\leq \frac{W(1)}{d_r^\rho}d_r^{\rho}=W(1)&\leq W(x)\\[2pt]
&\leq W(d_r)\leq W(d_r)\cdot x^\rho=c_2x^\rho,
\end{split} 
\end{equation*}
 since $W(x)$ and $x^{\rho}$ are an increasing functions. Hence, \eqref{eq:statement} is true for $n=0$. Now to the induction step: suppose that \eqref{eq:statement} holds and that $1\leq x\leq d_r\cdot d_1^{n+1}$. In particular, since we have established \eqref{eq:statement} for all $1\leq x\leq d_r$, we may assume without loss that $d_r\leq x\leq d_r\cdot d_1^{n+1}$. But then,\vspace{.05cm} 
\[1\leq \frac{d_r}{d_i}\leq\frac{x}{d_i}\leq d_r\cdot\frac{d_1}{d_i}\cdot d_1^{n}\leq d_r\cdot d_1^n \vspace{.1cm} \]
for all $1\leq i\leq r$. In particular, $1\leq x/d_i\leq d_r\cdot d_1^n$ and so we have that $c_1(x/d_i)^\rho\leq W(x/d_i)\leq c_2(x/d_i)^\rho$ by the induction hypothesis. Hence,  
\begin{equation*}
\begin{split}     
c_1x^{\rho}=c_1x^{\rho}\Big(\sum_{i=1}^r\frac{1}{d_i^\rho}\Big)=\sum_{i=1}^r c_1\Big(\frac{x}{d_i}\Big)^\rho&\leq\sum_{i=1}^rW\Big(\frac{x}{d_i}\Big)\\[3pt] 
&\leq \sum_{i=1}^r c_2\Big(\frac{x}{d_i}\Big)^\rho=c_2x^{\rho}\Big(\sum_{i=1}^r\frac{1}{d_i^\rho}\Big)=c_2x^{\rho}.
\end{split} 
\end{equation*}
However, \eqref{eq:functional} implies that $W(x)=\sum_{i=1}^r W(x/d_i)$, and we deduce that $c_1x^{\rho}\leq W(x)\leq c_2x^{\rho}$ for all $x\geq1$ as claimed.                          
\end{proof}
In particular, we immediately deduce the bounds in $\eqref{cor:functionount}$ for polarized sets of endomorphisms from Proposition \ref{thm:wordcount} and the fact that the degrees of endomorphisms are multiplicative with respect to composition.  
On the other hand, it is possible to obtain more exact asymptotics for the quantity in \eqref{wtatmostx} in the case of integer weights by analyzing the analytic properties of an associated generating function. With this in mind, we assume for the remainder of this section that the vector $\pmb{d}$ consists of positive \emph{integer} weights. Then for any $n\in\mathbb{Z}$, we let 
\begin{equation}\label{coeff}
a_{\pmb{d},n}:=\#\{\omega\in F_S\,:\, |\omega|_{\pmb{d}}=n\}
\end{equation} 
be the number of words of weight equal to $n$ and let
\begin{equation}\label{eq:dirichlet}  
\mathcal{D}_{\pmb{d}}(s):=\sum_{n\geq0} a_{\pmb{d},n}n^{-s} 
\end{equation} 
be the associated Dirichlet series generating function. In particular, we will improve upon Proposition \ref{thm:wordcount} in this case by applying known Tauberian theorems to \eqref{eq:dirichlet}. However, there is a catch: we must make two separate cases depending on whether the generating weights are all a powers of a single integer or not. With this in mind, we make the following definition.  
\begin{defin} A vector $\pmb{d}=(d_1,\dots, d_r)$ of positive integers is called \emph{cyclic} if there exists a single positive integer $d$ and a collection of positive integers $n_1,\ldots,n_r$ such that $d_j=d^{n_j}$ for all $j\geq1$. Moreover, if this is not the case, then we call $\pmb{d}$ \emph{acyclic}.  If $(d_1,\ldots ,d_r)$ is cyclic, then there is a unique positive integer $d$ such that $d_i=d^{a_i}$ for some positive integers $a_1,\ldots ,a_r$ with $\gcd(a_1,\ldots ,a_r)=1$, and we call this integer $d$ the \emph{base} of $\pmb{d}$.  
\end{defin} 
\begin{rem} It is an elementary fact of arithmetic that $\pmb{d}=(d_1,\dots, d_r)$ is cyclic if and only if the $d$'s are all pairwise multiplicatively dependent, and we used the second point of view in the introduction.        
\end{rem}  
With these notions in place, we now establish an exact asymptotic for the number of words of bounded multiplicative weight in free semigroups.   
\begin{thm}\label{thm:dirichlet}  Let $S=\{\alpha_1,\dots,\alpha_r\}$ be an alphabet, let $F_S$ be the free semigroup generated by $S$, and let $\pmb{d}=(d_1,\dots, d_r)$ be a vector of positive integers all at least $2$. Moreover, assume that $r\geq2$ and let $\rho_{\pmb{d}}>0$ be the unique real number satisfying $G_{\pmb{d}}(\rho_{\textbf{d}})=1$, where 
\[G_{\pmb{d}}(s)=d_1^{-s}+d_2^{-s}+\dots +d_r^{-s}. \vspace{.1cm} \] 
Then the following statements hold: \vspace{.15cm} 
\begin{enumerate} 
\item[\textup{(1)}] If $\pmb{d}$ is acyclic, then  
\[\mathcal{D}_{\pmb{d}}(s)=\frac{1}{1-G_{\pmb{d}}(s)} \vspace{.1cm}\]  
converges for all $\Re(s)\geq \rho_{\textbf{d}}$ except for a simple pole at $s=\rho_{\pmb{d}}$. Hence, the Wiener-Ikehara Tauberian theorem \cite[\S8.3]{Montgomery} implies that \vspace{.1cm}    
\[\#\{\omega\in F_S\,:\, |\omega|_{\pmb{d}}\leq X\}=\sum_{n\leq X}a_{\pmb{d},n}=\frac{-1}{\rho\, G'(\rho)}\,X^{\rho_{\pmb{d}}}+o(X^{\rho_{\pmb{d}}})\]
as $X$ tends to infinity. \vspace{.2cm} 
\item[\textup{(2)}] If $\pmb{d}$ is cyclic of base $d$, then there is a positive constant $C$ such that \vspace{.1cm} 
\[\#\{\omega\in F_S\,:\, |\omega|_{\pmb{d}}\leq X\}=C\cdot d^{\lfloor \log_d(X)\rfloor \cdot \rho}+o(d^{\lfloor \log_d(X)\rfloor \cdot \rho}) \vspace{.1cm}\] 
as $X$ tends to infinity.   
\end{enumerate}           
\end{thm} 
Before we can establish Theorem \ref{thm:dirichlet}, we need a few elementary facts, the first of which is a simple consequence of the strong triangle inequality. 
\begin{lem}\label{lem:fact1} Let $z_1,\dots ,z_r\in\mathbb{C}$ not all zero be such that 
\[z_1+\dots +z_n=|z_1|+\dots+|z_n|.\] 
Then $z_i=|z_i|$ for all $i$ (i.e., $\arg(z_i)=2k_i\pi$ for some $k_i\in\mathbb{Z}$).  
\end{lem}
Likewise, we need the following elementary fact from arithmetic. 
\begin{lem}\label{fact3} Let $d_1,\dots, d_r$ be positive integers all at least $2$. If the ratio $\log(d_i)/\log(d_j)\in\mathbb{Q}$ for all $i,j$, then the vector $\pmb{d}=(d_1,\dots,d_r)$ is cyclic. \vspace{.2cm} 
\end{lem} 
%
\begin{proof}[(Proof of Theorem \ref{thm:dirichlet})] 
For all $m\geq0$, let $F_S(m)$ be the set of sequences of length $m$ in $F_S$. Then it follows from the definition of weight in \eqref{eq:weight} that 
\[G_{\pmb{d}}(s)^m=\sum_{\omega\in F_S(m)}\hspace{-.3cm}|\omega|_{\pmb{d}}^{-s}=\sum_{n\geq1}\hspace{-.1cm}\sum_{\substack{\;|\omega|_{\pmb{d}}=n\\[2pt] \;\;\omega\in F_S(m)\\}}\hspace{-.4cm}n^{-s}=\sum_{n\geq1}a_{\pmb{d},n,m}\; n^{-s},\]
where $a_{\pmb{d},n,m}=\#\{\omega\in F_S(m)\,:\,|\omega|_{\pmb{d}}=n\}$. In particular, 
we see that 
\begin{equation*}
\begin{split} 
\frac{1}{1-G_{\pmb{d}}(s)}=\sum_{m\geq0} G_{\pmb{d}}(s)^m&=\sum_{m\geq0}\sum_{n\geq1}(a_{\pmb{d},n,m})\; n^{-s} \\[4pt] 
&=\sum_{n\geq1}\sum_{m\geq0}(a_{\pmb{d},n,m})\; n^{-s}=\sum_{n\geq1} a_{\pmb{d},n}n^{-s}=\mathcal{D}_{\pmb{d}}(s)
\end{split} 
\end{equation*}
holds formally. 
As for the convergence of this expression, note first that since the $d_i$ are positive, $G_{\pmb{d}}(s)$ is a strictly decreasing continuous function on the real line. Moreover, $G_{\pmb{d}}(0)=r\geq2$ and $\displaystyle{\lim_{x\rightarrow\infty} G_{\pmb{d}}(x)=0}$, and so there is a unique, simple, real solution $x=\rho_{\pmb{d}}>0$ to the equation $G_{\pmb{d}}(x)=1$. On the other hand, for all $\Re(s)>\rho_{\pmb{d}}$ we have that \vspace{.05cm} 
\begin{equation*}
\begin{split}
|G_{\pmb{d}}(s)|&=|d_1^{-s}+d_2^{-s}+\dots +d_r^{-s}|\\[3pt]
&\leq |d_1^{-s}|+|d_2^{-s}|+\dots+|d_r^{-s}|\\[3pt]
&=d_1^{-\Re(s)}+d_2^{-\Re(s)}+\dots+d_r^{-\Re(s)}\\[3pt]
&<d_1^{-\rho_{\pmb{d}}}+d_2^{-\rho_{\pmb{d}}}+\dots+d_r^{-\rho_{\pmb{d}}}\\[3pt]
&=1.
\end{split}
\end{equation*}
In particular, the meromorphic function $\mathcal{D}_{\pmb{d}}(s)=(1-G_{\pmb{d}}(s))^{-1}$ converges in $\Re(s)>\rho_{\pmb{d}}$ and has a simple pole at $s=\rho_{\pmb{d}}$. However, this is in fact the only pole on the line $\Re(s)=\rho_{\pmb{d}}$. To see this, suppose that $G_{\pmb{d}}(s)=1$ for some $\Re(s)=\rho_{\pmb{d}}$ and write $s=\rho_{\pmb{d}}+ci$ for some  $c\in\mathbb{R}$. Then setting $z_i=d_i^{-s}$, we deduce from Lemma \ref{lem:fact1} that \[\log(d_i)c=2\pi k_i\] 
for some $k_i\in\mathbb{Z}$. In particular, if $c\neq0$, then $\log(d_i)/\log(d_j)\in\mathbb{Q}$ for all $i,j$, and Lemma \ref{fact3} implies that $\pmb{d}$ is cyclic. Therefore, if $\pmb{d}$ is acyclic, then $\mathcal{D}_{\pmb{d}}(s)=(1-G_{\pmb{d}}(s))^{-1}$ converges for all $\Re(s)\geq \rho_{\textbf{d}}$, except for the simple pole at $s=\rho_{\pmb{d}}$. This completes the proof of statement (1). 

On the other hand, the cyclic case (say with base $d$) can be handled with somewhat simpler analysis. Write $d_i=d^{a_i}$ for some positive integers $a_i$ with $\gcd(a_1,\ldots ,a_r)=1$.  Since $F_S$ is free on $S=\{\alpha_1,\dots,\alpha_r\}$, there are exactly 
${n_1+\cdots + n_r \choose n_1,\ldots ,n_r}$ words with $n_i$ occurrences of $\alpha_i$ for all $i=1,\ldots ,r$.  Moreover, the weight of each such element is exactly 
$d^{\sum a_i n_i}$.  In particular, the number of elements in $F_S$ of weight $\le d^L$ is the sum of the coefficients of $x^i$ for $i=0,\ldots ,L$ in the polynomial 
\[\sum_{n\ge 0} (x^{a_1}+\cdots + x^{a_r})^n.\] 
In particular, we see that for a positive integer $L$, the number of elements in $F_S$ of weight $\le d^L$ is the coefficient of $x^L$ in the formal power series
\[F(x):=(1-x)^{-1} \cdot (1-x^{a_1}-\cdots - x^{a_r})^{-1}.\]
Observe that $G(x):=1-x^{a_1}-\cdots - x^{a_r}$ is strictly decreasing on $(0,\infty)$ and since the function is $0$ when $x=0$ and is negative when $x=1$, it has a unique zero $\theta$ in $(0,1)$. Since the derivative of $G$ is negative on $(0,\infty)$ this is a simple zero.  Observe further that $x=\theta$ is the unique pole of $F(x)$ in the closed disc of radius $\theta$. To see this, observe that if $\beta$ is another pole in this disc, then we must have $G(\beta)=0$.  Then 
\[|G(\beta)| \ge 1-|\beta|^{a_1}-\cdots -|\beta|^{a_r} \ge 1-\theta^{a_1}-\cdots -\theta^{a_r}=0,\] 
with equality only if $|\beta|=\theta$.  Thus it suffices to consider $\beta$ of the form $\theta e^{2\pi i\mu}$.  
Then $G(\beta) = 1- \theta^{a_1} \exp(2\pi i a_1 \mu) -\cdots -\theta^{a_r} \exp(2\pi i a_r \mu)$ and since $\sum \theta^{a_j} = 1$, this can only be zero if
$\exp(2\pi i a_1 \mu) =\cdots = \exp(2\pi i a_r \mu)=1$.  Since the $a_j$ have gcd 1, we then see this forces $\mu$ to be an integer and so $\beta=\theta$ as required. Using partial fractions and the fact that $F(x)$ has a simple pole at $x=\theta$ and no other poles in the closed disc of radius $\theta$, we see that 
$$F(x) = C/(1-x/\theta) + F_0(x)$$ 
for some positive constant $C$ and some rational function $F_0(x)$ whose radius of convergence is strictly larger than $\theta$.  In particular, the coefficient of $x^L$ in $F(x)$ is asymptotic to $C\theta^L$, and so we have that 
\[\#\{\omega\in F_S\,:\, |\omega|_{\pmb{d}}\leq d^L\}\sim C\theta^L.\]
Since $\theta$ is the unique positive solution to $\sum \theta^{a_i} = 1$, we see that if 
$\rho$ is such that $d^{\rho}=\theta$, then $\rho$ is the unique positive solution
to $\sum 1/d_i^{\rho}=1$ and 
\[\#\{\omega\in F_S\,:\, |\omega|_{\pmb{d}}\leq d^L\}\sim C (d^L)^{\rho}.\]
Now for $X$ large, there is a unique $L$ such that $d^L \le X < d^{L+1}$. Hence, in general we have that
\[\#\{\omega\in F_S\,:\, |\omega|_{\pmb{d}}\leq X\} \sim C d^{\lfloor \log_d(X)\rfloor \cdot \rho},\]
a fact equivalent to statement (2).      
\end{proof}

In particular, we obtain an improved asymptotic for the number of functions of bounded degree in free semigroups generated by polarized sets. 
    
\begin{cor}\label{cor:acyclicdegrees} Let $V$ be a projective variety, let $S:=\{\phi_1,\dots,\phi_r\}$ be a polarized set of endomorphisms on $V$, and say $\phi_i^*\mathcal{L}\cong\mathcal{L}^{\otimes d_i}$ for some integers $d_i\geq2$. Moreover, assume that $r\geq2$, that $M_S$ is free, and that $\rho$ is the real number satisfying $d_1^{-\rho}+\dots+d_r^{-\rho}=1$.  
Then the following statements hold.\vspace{.1cm}  
\begin{enumerate} 
\item If $\pmb{d}=(d_1,\dots, d_r)$ is acyclic there is a positive constant $c$ such that 
\[\#\{f\in M_S\,:\, \deg_\mathcal{L}(f)\leq X\}\sim cX^{\rho}. \vspace{.1cm}\] 
\item If $\pmb{d}=(d_1,\dots, d_r)$ is cyclic of base $d$ there is a positive constant $C$ such that  
\[ \#\{f\in M_S\,:\, \deg_\mathcal{L}(f)\leq X\}\sim C\cdot d^{\lfloor \log_d(X)\rfloor \cdot \rho}\]     
\end{enumerate} 
\end{cor} 
\begin{rem} We note that our results on degrees in this section work over an arbitrary ground field. However, for the remainder of this paper (when we instead count by heights), we will restrict ourselves to number fields. 
\end{rem}
\subsection{From degrees to heights} 
Now that we have a handle on the growth of degrees in polarized  semigroups, we use this to study the growth of heights in orbits. To do this, we fix some notation. As before let $S=\{\phi_1,\dots,\phi_r\}$ be a polarized set of endomorphisms on a projective variety $V$ with $\phi_i^*\mathcal{L}\cong\mathcal{L}^{\otimes d_i}$ for some $d_i>1$ and some ample line bundle $\mathcal{L}$. Moreover, we may choose a height function $h_{\mathcal{L}}: V\rightarrow\mathbb{R}_{\geq0}$ associated to $\mathcal{L}$, well-defined up to $O(1)$. Then the basic tool that makes the conversion between degrees and heights possible is the following fact: for each $\phi_i$ there is a constant $C_i$ such that
\begin{equation}\label{hts:functoriality}
\Big|h_{\mathcal{L}}(\phi_i(P))-d_i\, h_{\mathcal{L}}(P)\Big|\leq C_i
\end{equation} 
holds for all $P\in V$; see, for instance, \cite[Theorem 7.29]{SilvDyn}. In particular, one may iterate the bound above and use a telescoping argument (due to Tate) to obtain a similar statement for semigroups; see \cite[Lemma 2.1]{Wade+Viv}.    
\begin{lem}\label{Tate} Let $S=\{\phi_1,\dots, \phi_r\}$ be a polarized set of endomorphisms on a variety $V$, let $d_S=\min\{d_i\}$, and let $C_S=\max\{C_{i}\}$ be as in \eqref{hts:functoriality}. Then \vspace{.05cm}
\[\bigg|\frac{h_{\mathcal{L}}(f(P))}{\deg_{\mathcal{L}}(f)}-h_{\mathcal{L}}(P)\bigg|\leq \frac{C_S}{d_S-1} \vspace{.05cm}\] 
holds for all $f\in M_S$ and all $P\in V$. 
\end{lem}
In particular, we may deduce the following result, which says first that the set of preperiodic points of $M_S$ is a set of bounded height and second that if $P$ is not a preperiodic point, then $h(f(P))\rightarrow\infty$ as $|f|\rightarrow\infty$.    
\begin{lem}\label{lem:hts} Let $S=\{ \phi_1,\ldots ,\phi_r\}$ be a finite set of endomorphisms of a variety $V$.  Then there is a positive constant $C$ such that: \vspace{.1cm} 
\begin{enumerate}
\item if $P$ is a preperiodic point for the semigroup $M_S$ then $h_{\mathcal{L}}(P)\le C$; \vspace{.15cm} 
\item if $h_{\mathcal{L}}(P) > C$ then $h_{\mathcal{L}}(f(P)) > h_{\mathcal{L}}(P)$ for all $f\in M_S$. \vspace{.15cm} 
\end{enumerate}
In particular, if $P$ is not preperiodic for $M_S$ and $B>0$, then there is an $N:=N(V,S,P,B)$ such that $h(f(P))>B$ for all $f\in M_S$ with $|f|>N$. 
\end{lem}
\begin{proof}
Notice that if $P$ is a preperiodic point for $M_S$ and $Q:=f(P)=g\circ f(P)$ then 
\[\bigg|\frac{h_{\mathcal{L}}(g(Q))}{\deg_{\mathcal{L}}(g)}-h_{\mathcal{L}}(Q)\bigg|\leq \frac{C_S}{d_S-1} \vspace{.05cm}\] and since $g(Q)=Q$ and $\deg_{\mathcal{L}}(g)\ge 2$ we then see that 
$$h_{\mathcal{L}}(Q)\le 2C_S\cdot (d_S-1)^{-1}.$$  Similarly, we have
\[\bigg|\frac{h_{\mathcal{L}}(f(P))}{\deg_{\mathcal{L}}(f)}-h_{\mathcal{L}}(P)\bigg|\leq \frac{C_S}{d_S-1} \vspace{.05cm}\]
and since $f(P)=Q$ we use the inequality on the height of $Q$ to get that $h_{\mathcal{L}}(P)\le 2\frac{C_S}{d_S-1}$.
Also, observe that if $h_{\mathcal{L}}(P)\ge 2C_S$ then 
if $f\in M_S$ then Lemma \ref{Tate} gives that 
$h_{\mathcal{L}}(f(P))/\deg_{\mathcal{L}}(f)> h_{\mathcal{L}}(P)/2$ and so 
$h_{\mathcal{L}}(f(P))>h_{\mathcal{L}}(P)$ for all $f\in M_S$.  
In particular, setting $C=2C_S$, we deduce both claims (1) and (2).

For the final claim, assume that $P$ is not preperiodic for $M_S$ and that $B>0$. Note that without loss of generality, we may assume that $B\geq C$. Now let $K$ be a number field over which $P$ and every map in $S$ is defined. Then $f(P)\subseteq V(K)$ for all $f\in M_S$ and   
\[N(B):=\#\{Q\in V(K)\,:\, h_\mathcal{L}(P)\leq B\}\]
is finite by Northcott's theorem. Now assume that $|f|=m>N$ and write $f=\theta_m\circ\theta_{m-1}\circ\dots\circ\theta_1$ for some $\theta_i\in S$. Likewise, for $1\leq i\leq m$ let $f_i:=\theta_i\circ\theta_{i-1}\circ\dots\circ\theta_1$. Note first that if $h(f_i(P))>B\geq C$ for some $i$, then $h(f(P))>B$ also by property (2). Therefore, we may assume that $h(f_i(P))\leq B$ for all $i\leq m$. On the other hand $m>N$, so that $f_i(P)=f_j(P)$ for some $i>j$ by the pigeon-hole principle and the definition of $N$. But then 
\[g_{ij}\circ f_j(P)=f_i(P)=f_j(P)\]
for $g_{ij}=\theta_i\circ\dots\circ \theta_{j+1}$, a contradiction of our assumption that $P$ is not preperiodic. Therefore, $h(f(P))>B$ as claimed.         
\end{proof} 
Lastly, before proving Theorem \ref{thm:functioncount}, we need the following result, which provides a key new ingredient for improving the estimates in \cite{Wade:MathZ}.   
\begin{lem}\label{lem:beta}
Suppose that $M_S$ is free and that $P$ is a not a preperiodic point of $V$.  Then there is a positive constant $\beta=\beta(S,\mathcal{L},P)$ such that 
$$\lim_{n\rightarrow\infty}\sum_{g\in S_n} h_{\mathcal{L}}(g(P)))^{-\rho}=\beta.$$
\end{lem}
\begin{proof}
For $n$ let $\beta_n=\beta_n(P)=\sum_{g\in S_n} h_{\mathcal{L}}(g(P))^{-\rho}$.  Since $M_S$ is free, $S_{n+1}$ is the disjoint union of $\phi_1\circ S_n,\ldots , \phi_r\circ S_n$.  
Then
$$\left| h_{\mathcal{L}}(\phi_i(f(P))) - d_i h_{\mathcal{L}}(f(P))\right| \le C_S$$ and so
$$d_i h_{\mathcal{L}}(f(P))-C_S\le  h_{\mathcal{L}}(\phi_i(f(P))) \le d_i h_{\mathcal{L}}(f(P))+C_S.$$
In particular, the mean value theorem gives the inequality
$$\left| h_{\mathcal{L}}(\phi_i(f(P)))^{-\rho} - (d_i  h_{\mathcal{L}}(f(P))^{-\rho}\right| \le C_S \rho\, \big| d_i\cdot h_{\mathcal{L}}(f(P))-C_S \big|^{-\rho-1}.$$
Assume first that $h(P)>C:=2C_S$. Then since $\sum_{i=1}^r 1/d_i^{\rho}=1$, we have 
\begin{eqnarray*} |\beta_{n+1}-\beta_n| &= & \left|\sum_{i=1}^r \sum_{f\in S_n}  h_{\mathcal{L}}(\phi_i\circ f(P))^{-\rho} - (d_ih_{\mathcal{L}}(f(P)))^{-\rho} \right| \\[4pt]
&\le&  \sum_{i=1}^r \sum_{f\in S_n}
C_S \rho\, \big|d_i\cdot h_{\mathcal{L}}(f(P))-C_S\big|^{-\rho-1}\\[4pt]
&\le&  \sum_{i=1}^r \sum_{f\in S_n}
 \frac{C_S\rho}{(C_S( d_i {\rm deg}_{\mathcal{L}}(f)-1))^{\rho+1}}\\[4pt]
 &\le &  C_S^{-\rho} 2^{\rho+1}\rho \sum_{i=1}^r \sum_{f\in S_n} 1/(d_i {\rm deg}_{\mathcal{L}}(f))^{\rho+1}\\[4pt]
 &=& C_S^{-\rho} 2^{\rho+1}\rho \sum_{j_1+\cdots +j_r=n+1} {n+1\choose j_1,\ldots ,j_r} d_1^{-j_1(\rho+1)} \cdots d_r^{-j_r(\rho+1)} \\[4pt] 
 &=& C_S^{-\rho} 2^{\rho+1}\rho\cdot (1/d_1^{\rho+1} + \cdots + 1/d_r^{\rho+1})^{n+1}.\\ 
\end{eqnarray*}
In particular, since 
$C':=\sum 1/d_i^{\rho+1} < 1$ and since $\sum_j (C')^j$ converges, we see that $\{\beta_n(P)\}$ is a Cauchy sequence converging to some $\beta(P)$ as claimed. This completes the proof when $h(P)\geq2C_S$. For the general case, assume only that $P$ is not preperiodic for $S$. Then Lemma \ref{lem:hts} implies that there exists $N$ such that $h(f(P))>2C_S$ for all $|f|\geq N$. On the other hand, for $n>N$ we have that 
\[\sum_{g\in S_n} h_{\mathcal{L}}(g(P))^{-\rho}=\sum_{f\in S_{N}}\sum_{F\in S_{n-N}}h_{\mathcal{L}}(F(f(P))^{-\rho}=\sum_{f\in S_N}\beta_{n-N}(f(P))\] 
since $M_S$ is free. Moreover, for each $f\in M_S$ the sequence $\{\beta_{n-N}(f(P))\}$ converges by our argument above. Therefore, $\{\beta_n(P)\}$, a finite sum of convergent sequences, must converge also.     
\end{proof}

In particular, we may use Lemma \ref{Tate}, Lemma \ref{lem:beta}, and the degree growth estimates in Section \ref{sec:degrees} to bound the number of functions in a polarized semigroup yielding a bounded height relation; see Theorem \ref{thm:functioncount} from the Introduction.  
\begin{proof}[(Proof of Theorem \ref{thm:functioncount})] Assume that $r=\#S\geq2$, let $\rho$ be the unique positive number satisfying ${d_1}^{-\rho}+\dots+{d_r}^{-\rho}=1$, and let  $b_S=C_S/(d_S-1)$; here $d_S$ and $C_S$ are the constants in Lemma \ref{Tate}. We assume first that $P\in V$ satisfies $h_{\mathcal{L}}(P)>b_S$ and then show how the general non-preperiodic case follows from this one. In particular, if $h_{\mathcal{L}}(P)>b_S$, then Lemma \ref{Tate} implies  
\begin{equation*}\label{subsets}
\begin{split} 
\hspace{-.1cm}\bigg\{f\in M_S:\, \deg_{\mathcal{L}}(f)\leq \frac{X}{h_{\mathcal{L}}(P)+b_s}\bigg\}&\subseteq \bigg\{f\in M_S:\, h_{\mathcal{L}}(f(P))\leq X\bigg\}\\[5pt] 
&\subseteq \bigg\{f\in M_S:\, \deg_{\mathcal{L}}(f)\leq \frac{X}{h_{\mathcal{L}}(P)-b_s}\bigg\}\vspace{.1cm}
\end{split} 
\end{equation*}
In particular, Proposition \ref{thm:wordcount} implies that there is a constant $c$ such that 
\begin{equation*}
\begin{split} 
\#\{f\in M_S\,:\, h_{\mathcal{L}}(f(P))\leq X\}&\leq\#\bigg\{f\in M_S\,:\, \deg_{\mathcal{L}}(f)\leq \frac{X}{h_{\mathcal{L}}(P)-b_s}\bigg\}\\[5pt] 
&\leq \#\bigg\{\omega\in F(S)\,:\, |\omega|_{\pmb{d}} \leq \frac{X}{h_{\mathcal{L}}(P)-b_s}\bigg\}\\[5pt]
&\leq c\Big(\frac{X}{h_{\mathcal{L}}(P)-b_s}\Big)^\rho \\[5pt] 
\end{split} 
\end{equation*}
holds for all $X$ sufficiently large; here $\pmb{d}=(d_1,\dots,d_r)$ is the vector of degrees of functions in $S$. Hence, 
we obtain statement (1) as claimed. On the other hand, if $M_S$ contains a free semigroup with finite basis $S'\subset M_S$, let $\rho'$ be the unique positive number satisfying $\sum_{g\in S'}1/\deg(g)^{\rho'}=1$. Then Proposition \ref{thm:wordcount} and Lemma \ref{Tate} imply that there is a positive constant $c'$ such that 
\begin{equation*} \label{middlecase}
\begin{split} 
\#\{f\in M_S\,:\, h_{\mathcal{L}}(f(P))\leq X\}&\geq\#\bigg\{f\in M_S\,:\, \deg_{\mathcal{L}}(f)\leq \frac{X}{h_{\mathcal{L}}(P)+b_s}\bigg\}\\[5pt] 
&\geq \#\bigg\{f\in M_{S'}\,:\, \deg_{\mathcal{L}}(f)\leq \frac{X}{h_{\mathcal{L}}(P)+b_s}\bigg\}\\[5pt]
&\geq c'\,\Big(\frac{X}{h_{\mathcal{L}}(P)+b_s}\Big)^{\rho'}\\[5pt] 
\end{split} 
\end{equation*} 
holds for all $X$ sufficiently large. In particular, we deduce statement (2) as claimed; moreover, clearly $\rho'\leq \rho$ since $X^{\rho'}\ll X^{\rho}$ by combining the bounds in statements (1) and (2). 

Next, suppose that $M_S$ is itself free. Then taking $S=S'$ in the proof of statement (2) above, we see that $\#\{f\in M_S\,:\, h_{\mathcal{L}}(f(P))\leq X\}\asymp X^\rho$ for all sufficiently large $X$ as claimed in statement (3).  

Finally, suppose that $M_S$ is free and that the vector of degrees $\pmb{d}=(d_1,\dots,d_r)$ is acyclic. Moreover, let $\epsilon>0$ and let 
$$\beta:=\lim_{n\to\infty} \sum_{g\in S_n} h_{\mathcal{L}}(g(P))^{-\rho}$$ 
be the constant from Lemma \ref{lem:beta}. Then Lemma \ref{lem:hts} implies that $h_{\mathcal{L}}(g(P))\to\infty$ as $|g|\to\infty$ and so there is a natural number $n$ such that
\[\left| \sum_{g\in S_n} h_{\mathcal{L}}(g(P))^{-\rho}-\beta\right| < \epsilon\] 
and
\[h_{\mathcal{L}}(g(P))-b_s > (1-\epsilon) h_{\mathcal{L}}(g(P)) \vspace{.1cm} \] 
for all $g\in S_n$. Then statement (1) of Corollary \ref{cor:acyclicdegrees} implies that for all $g\in S_n$ and all large $X$ we have that  \vspace{.1cm} 
\begin{align*}
&~ \#\{f\in M_S\,:\, h_{\mathcal{L}}(f\circ g(P))\leq X\}\\[6pt] 
&\leq  \#\bigg\{f\in M_S\,:\, \deg_{\mathcal{L}}(f)\leq \frac{X}{h_{\mathcal{L}}(g(P))-b_s}\bigg\}\\[6pt] 
&\leq  c(1+\epsilon)\Big(\frac{X}{h_{\mathcal{L}}(g(P))-b_S}\Big)^{\rho} \\[6pt]  
&\le c(1+\epsilon)(1-\epsilon)^{-\rho} X^{\rho}/h_{\mathcal{L}}(g(P))^{\rho}. 
\end{align*}
On the other hand, for large $X$ we have that 
\[\#\{f\in M_S\,:\, h_{\mathcal{L}}(f(P))\leq X\} = \#S_{<n} + \sum_{g\in S_n} \#\{f\in M_S\,:\, h_{\mathcal{L}}(f\circ g(P))\leq X\}\] 
and so we get that  
$$\#\{f\in M_S\,:\, h_{\mathcal{L}}(f(P))\leq X\} \le \#S_{<n} + c(1+\epsilon)(1-\epsilon)^{\rho} X^{\rho}(\beta+\epsilon)$$ 
for $X$ sufficiently large. Moreover, by a completely analogous argument we get a lower bound of
$$\#\{f\in M_S\,:\, h_{\mathcal{L}}(f(P))\leq X\} \ge \#S_{<n} + c(1-\epsilon)(1+\epsilon)^{\rho} X^{\rho}(\beta-\epsilon)$$
for all large $X$. Then, since $n$ is fixed and $\epsilon$ is arbitrarily small, we see that 
$$\#\{f\in M_S\,:\, h_{\mathcal{L}}(f(P))\leq X\}= c\beta X^{\rho}(1+o(1)).$$
 as claimed in statement (4). 

Finally, assume now only that $P$ is a non-preperiodic for $M_S$. Then Lemma \ref{lem:hts} implies that there exists $N$ such that $h(g(P))>b_S$ for all $|g|\geq N$. In particular, for all large $X$ we have that 
\[
\scalemath{0.95}{
\#\{f\in M_S\,:\, h_{\mathcal{L}}(f(P))\leq X\}=\#S_{\leq N}+\sum_{g\in S_N}\#\{f\in M_S\,:\, h_{\mathcal{L}}(f(g(P)))\leq X \}.}
\]
However, by replacing $P$ with $g(P)$ in our arguments above, we see that each of the finitely many terms $\#\{f\in M_S\,:\, h_{\mathcal{L}}(f(g(P)))\leq X \}$ for $g\in S_N$ exhibit growth according to statements (1) through (4) of Theorem \ref{thm:functioncount}, assuming the corresponding hypotheses. In particular, the total count $\#\{f\in M_S\,:\, h_{\mathcal{L}}(f(P))\leq X\}$ exhibits the desired growth also.      
\end{proof} 
Lastly, we note that our upper bound on the number of functions $f\in M_S$ such that $h_\mathcal{L}(f(P))\leq X$ for non-preperiodic basepoints $P$ leads to a general upper bound for the number of points of bounded height in orbits; see Theorem \ref{thm:universalupperbd} from the Introduction. The last step is the following lemma: 
\begin{lem}\label{lem:general} Let $S=\{\phi_1,\dots,\phi_r\}$ be a polarizable set of endomorphisms on a variety $V$, let $P\in V$ be such that $\Orb_S(P)$ is infinite, and let $B>0$. Then there is a finite subset $F\subseteq\Orb_S(P)$ and distinct points $Q_1,\dots, Q_n$ in $\Orb_S(P)$ such that the following statements hold: \vspace{.1cm}  
\begin{enumerate} 
\item Each $Q_i$ satisfies $h_{\mathcal{L}}(Q_i)>B$. \vspace{.15cm}  
\item $\Orb_S(P)=F\cup\Orb_S(Q_1)\cup\dots\cup\Orb_S(Q_{n})$.\vspace{.1cm}  
\end{enumerate}
In particular, we may assume that $Q_1,\dots,Q_n$ are not preperiodic for $S$.  
\end{lem}
\begin{proof} Let $C$ be the constant in Lemma \ref{lem:hts}. As a reminder $C=2C_S$, where $C_S$ is the constant from Lemma \ref{Tate}. Moreover, let $d=\max\{d_i\}$ where $d_i$ is as in \eqref{hts:functoriality}. Now let $P\in V$ have infinite orbit and $B>0$. Without loss, we may assume $B\geq C$. Note that the desired conclusion follows easily if $h_\mathcal{L}(P)>B$ by taking $n=1$, $P=Q_1$, and $F=\varnothing$. Therefore, we may assume that $h_{\mathcal{L}}(P)\leq B$. Now define the set 
\[T:=\big\{Q\in\Orb_S(P)\,:\, B< h_{\mathcal{L}}(Q)\leq2dB\big\}.\] 
In particular, $T=Q_1,\dots, Q_n$ is a finite (possibly empty) set by Northcott's theorem (seen by working over a fixed number field $K$ over which $P$ and every map in $S$ is defined). Likewise, $F:=\{Q\in\Orb_S(P)\,:\, h_{\mathcal{L}}(Q)\leq B\}$ is a finite set, and we will show that $F$ and the points $Q_i$ have the desired properties. To see this, suppose that $Q=\theta_m\circ\dots\circ\theta_1(P)$ for some $\theta_i\in S$ satisfies $h_{\mathcal{L}}(Q)>B$ (such $Q$ must exist since $\Orb_S(P)$ is infinite), and define
\[m_0=\min\big\{s\,:\, h_{\mathcal{L}}(\theta_s\circ\dots\circ\theta_1(P))>B\big\}.\vspace{.1cm} \]
Note that $m_0$ exists since $m$ is in the defining set above. Moreover, $m_0\geq1$ since $h_{\mathcal{L}}(P)\leq B$ by assumption. We claim that $Q':=\theta_{m_0}\circ\dots\circ\theta_1(P)\in T$ (in particular, $T$ is non-empty): if not, then necessarily $h_{\mathcal{L}}(Q')>2d B$. However, we then see that
\begin{equation*}
\begin{split} 
2dB&<h_{\mathcal{L}}(Q')=h_{\mathcal{L}}(\theta_{m_0}(\theta_{m_0-1}\circ\dots\circ\theta_1(P)))\\[3pt] 
&\leq\deg(\theta_{m_0})(h_{\mathcal{L}}(\theta_{m_0-1}\circ\dots\circ\theta_1(P))+C_S)\leq2dB,
\end{split} 
\end{equation*}
a contradiction; here we use Lemma \ref{Tate} and the minimality of $m_0$. In particular, we have shown that for any $Q\in\Orb_S(P)$ of height bigger than $B$ there is a point $Q_i\in T$ such that $Q\in\Orb_S(Q_i)$. Hence, we deduce that $\Orb_S(P)=F\cup\Orb_S(Q_1)\cup\dots\cup\Orb_S(Q_{n})$ as claimed.      
\end{proof}
In particular, since every point in $\Orb_S(P)$ of height at most $X$ is determined by at least one function $f\in M_S$ satisfying $h_{\mathcal{L}}(f(P))\leq X$ and (outside of a finite set) all orbits are covered by non-preperiodic orbits, we see that 
\[\#\{Q\in\Orb_S(P)\,:\, h_{\mathcal{L}}(Q)\leq X\}\ll X^{\rho}\]
holds for all $P\in V$ by combining Theorem \ref{thm:functioncount} part (1) with Lemma \ref{lem:general}.  
\section{Orbit counts in $\mathbb{P}^1$}\label{sec:P^1}
\subsection{Generic bounds} Now that we have some handle on the growth of functions with various boundedness properties in a semigroup, we turn to the asymptotic growth of points in orbits of bounded height. To make this transition, we need some control on the functions that agree at a particular value. This is possible when $V=\mathbb{P}^1$ precisely because we have some control on the set of rational (or integral) points on curves.  
To state this refinement in dimension $1$, it will be useful to have the following general definition.
\begin{defn}\label{def:cancel}
     Let $S$ be a polarizable set of rational maps on a variety $V$  and let $P\in V$. Then we say that $M_S$ has the \emph{cancellation property} if there is a positive constant $B=B(P)$ such that if $f(Q_1)=g(Q_2)$ for some $f,g\in M_S $ and some $Q_1, Q_2 \in \Orb_S(P)$ with $h_{\mathcal{L}}(Q_i)> B$, then $f=\phi\circ f_0$ and $g=\phi\circ g_0$ for some $\phi\in S$ and $f_0,g_0\in M_S$. Moreover, $f_0(Q_1)=g_0(Q_2)$.   
\end{defn}
In particular, if $M_S$ has the cancellation property above, then an asymptotic on the number of functions yielding a bounded height relation leads to an asymptotic for the number of points of bounded height in orbits. 
\begin{prop}\label{prop:sum}
Let $S=\{\phi_1,\dots,\phi_r\}$ be a polarizable set of rational maps on a variety $V$ generating a free semigroup with the cancellation property and suppose that $P\in V$ is such that $\Orb_S(P)$ is infinite. Then there exists a $B>0$ and points $Q_1,\ldots ,Q_{t}\in\Orb_S(P)$ such that $h(Q_i)>B$ and \vspace{-.2cm} 
\begin{equation}\label{eq:sum}
\scalemath{.93}{
\#\{Q\in {\rm Orb}_S(P) \colon h_{\mathcal{L}}(Q) \le X\}\, \sim
\sum_{i=1}^{t} \#\{f\in M_S\,:\, h_{\mathcal{L}}(f(Q_i)) \le X\}.}
\end{equation} 
In particular, we may choose non-preperiodic points $Q_1,\ldots ,Q_{t}\in\Orb_S(P)$ such that \eqref{eq:sum} holds.  
\end{prop}
\begin{proof}
    Let $B_0$ be the constant from Lemma \ref{lem:hts}, let $B_1$ be the constant from Definition \ref{def:cancel}, and let $B = \max(B_1, B_0)$. Then Lemma \ref{lem:general} implies that there exists a finite set $F$ and points $Q_1, \dots, Q_n$ satisfying $h_{\mathcal{L}}(Q_i) > B$ such that 
    \begin{equation}
        \Orb_S(P) = F \cup \bigcup^n_{i = 1} \Orb_S(Q_i).
    \end{equation}
    We will show that the union over the $\Orb_S(Q_i)$ may be written as a disjoint union. To see this, suppose that $f(Q_i) = g(Q_j)$ for some $f, g \in M_S$. If $|f|=|g|$, then repeated application of the cancellation property implies that $Q_i=Q_j$. On the other hand if $|f| \neq |g|$, say without loss $|f| > |g|$, then repeated application of the cancellation property implies that $f_0(Q_i) = Q_j$ for some $f_0 \in M_S$. Hence, $\Orb_S(Q_j) \subseteq \Orb_S(Q_i)$ in this case. Thus, after reordering if necessary, we may choose $t\leq n$ such that \begin{equation}\label{eq: orbdisjointdecomp}
        \Orb_S(P) = F \cup \bigsqcup^t_{i = 1} \Orb_S(Q_i).
    \end{equation}
    On the other hand, for each $Q_i$ a similar argument implies that the evaluation map $f\rightarrow f(Q_i)$ is injective: if $f(Q_i) = g(Q_i)$ for some $|f| = |g|$, then repeated application of the cancellation property implies that $f = g$. Likewise, if $|f|>|g|$ then repeated application of the cancellation property implies that $f_0(Q_i) = Q_i$ for some non-identity map $f_0\in M_S.$ But then $Q_i$ is preperiodic, contradicting Lemma \ref{lem:hts} and the fact that $h_{\mathcal{L}}(Q_i) > B$. Therefore, for each $Q_i$ the evaluation map $f\rightarrow f(Q_i)$ is injective, and thus
    \begin{equation}\label{eq:nonporbcunt}
        \# \{Q \in \Orb_S(Q_i): h_{\mathcal{L}}(Q) \leq X\} = \#\{f \in M_S : h_{\mathcal{L}}(f(Q_i)) \leq X\}. 
    \end{equation}
    Hence, combining (\ref{eq: orbdisjointdecomp}) and (\ref{eq:nonporbcunt}), we obtain (\ref{eq:sum}) as claimed. 
\end{proof} 
We next note that many sets of rational maps on $V=\mathbb{P}^1$ generate semigroups with the cancellation property above. To make this precise, recall that $w\in\mathbb{P}^1(\overline{\mathbb{Q}})$ is called a \emph{critical value} of $\phi\in\overline{\mathbb{Q}}(x)$ if $\phi^{-1}(w)$ contains fewer than $\deg(\phi)$ elements. Likewise, we call a critical value $w$ of $\phi$ \emph{simple} if $\phi^{-1}(w)$ contains exactly $\deg(\phi)-1$ points. In particular, we have the corresponding notions for sets:   
\begin{defin} Let $S=\{\phi_1,\dots,\phi_r\}$ be a set of rational maps on $\mathbb{P}^1$ and let $\mathcal{C}_{\phi_i}$ denote the set of critical values of $\phi_i$. Then $S$ is called \emph{critically separate} if $\mathcal{C}_{\phi_i}\cap\mathcal{C}_{\phi_j}=\varnothing$ for all $i\neq j$. Moreover, $S$ is called \emph{critically simple} if every critical value of every $\phi\in S$ is simple.  
\label{def:crit}
\end{defin}  
In particular, critically separate and critically simple sets of maps have the cancellation property; see also Proposition 4.1 and Lemma 4.8 in \cite{Wade:MathZ}.          
\begin{prop}\label{thm:fstops} Let $S=\{\phi_1,\dots,\phi_r\}$ be a critically separate and critically simple set of rational maps on $\mathbb{P}^1$ all of degree at least four. The  following statements hold: \vspace{.1cm} 
\begin{enumerate} 
\item[\textup{(1)}] $M_S$ is a free semigroup.\vspace{.2cm}     
\item[\textup{(2)}] For all non-preperiodic $P$ there is a constant $t_{S,P}$ such that
\[\#\{f\in M_S\,:\, f(P)=Q\}\leq t_{S,P}\]
holds for all $Q\in\Orb_S(P)$. \vspace{.2cm} 
\item[\textup{(3)}] The semigroup $M_S$ has the cancellation property.    
\end{enumerate}  
\label{prop:can}
\end{prop} 
\begin{proof} The first two statements are proved in \cite[\S4]{Wade:MathZ}. Therefore, it only remains to prove (3). 
By the main results in \cite{Pakovich} (see also \cite[Proposition 4.6]{Wade:MathZ}), we see that an irreducible component of a curve of the form
$(\phi_i,\phi_j)^{-1}(\Delta)\subseteq\mathbb{P}^1\times\mathbb{P}^1$ is either $\Delta$ or an irreducible curve of geometric genus $\ge 2$, where $\Delta$ is the diagonal subvariety.  Thus by Faltings' theorem the set $\Sigma$ of $K$-points lying on some component other than $\Delta$ is finite. Let 
\[B_0 = \max_{(P_1,P_2) \in \Sigma}\big\{\max\{h_{\mathcal{L}}(P_1),h_{\mathcal{L}}(P_2)\}\big\},\] 
let $B_1$ be the constant from Lemma \ref{lem:hts}, and let $B = \max\{B_0, B_1\}$. Now suppose that $f,g\in M_S$ and $Q_1, Q_2 \in \P^1 $ satisfy 
\[\min(h_{\mathcal{L}}(Q_1), h_{\mathcal{L}}(Q_2)) > B\] 
and $f(Q_1)=g(Q_2)$. We will show that $B$ satisfies the conditions needed in Definition \ref{def:cancel}. To see this, write $f=\phi_i\circ f_0$ and $g=\phi_j\circ g_0$ for  $f_0,g_0\in M_S$ and $1\leq i,j\leq r$. Note that if either $i\neq j$ or $i=j$ and $f_0(Q_1)\neq g_0(Q_2)$, then 
$(f_0(Q_1),g_0(Q_2))$ lies on an irreducible component of $(\phi_i,\phi_j)^{-1}(\Delta)$ that is not equal to $\Delta$. Hence, $(f_0(Q_1),g_0(Q_2))$ is a point on $\Sigma$, and so \[\max(h_{\mathcal{L}}(f_0(Q_1)),h_{\mathcal{L}}(g_0(Q_2)) )\leq B_0\leq B.\] 
On the other hand, since  $\min(h_{\mathcal{L}}(Q_1), h_{\mathcal{L}}(Q_2)) > B \geq B_1$, it follows from Lemma \ref{lem:hts} that 
\[\min(h_{\mathcal{L}}(f_0(Q_1)),h_{\mathcal{L}}(g_0(Q_2)) ) > B,\] 
a contradiction. Therefore, $i=j$ and $f_0(Q_1)\neq g_0(Q_2)$ as desired.  
\end{proof}

We can now improve the main result in \cite{Wade:MathZ}. 
\begin{thm}\label{thm:rational}  Let $S=\{\phi_1,\dots,\phi_r\}$ be a critically separate and critically simple set of rational maps on $\mathbb{P}^1$ all of degree at least four. Moreover, assume that $r\geq2$ and let $\rho$ be the positive real number satisfying $\sum \deg(\phi_i)^{-\rho}=1$. Then 
\[\#\{Q\in\Orb_S(P)\,:\, h(Q)\leq X\}\asymp X^{\rho}\]
for all $P\in\mathbb{P}^1(\overline{\mathbb{Q}})$ such that $\Orb_S(P)$ is infinite. Moreover, if the degrees of the maps in $S$ are acyclic, then
\begin{equation*}
\#\{Q\in\Orb_S(P)\,:\, h(Q)\leq X\}\sim cX^{\rho}
\end{equation*} 
for some positive constant $c$, depending on $P$ and $S$. \end{thm} 
\begin{proof} Assume that $P$ is a point in $\P^1$ satisfies that $\Orb_S(P)$ is infinite. Then Proposition \ref{thm:fstops} implies that $M_S$ is free and has the cancellation property. Hence, Proposition \ref{prop:sum} implies that there exists non-preperiodic points $Q_1,\ldots ,Q_t\in\Orb_S(P)$ such that 
\begin{equation}\label{step1}
\scalemath{0.95}{
\#\{Q\in {\rm Orb}_S(P) \colon h(Q) \le X\} \sim \
\sum_{i=1}^t \#\{f\in M_S\,:\, h(f(Q_i)) \le X\}.}
\end{equation} 
From here, the result now follows directly from Theorem \ref{thm:functioncount}. For instance, if the the degrees of the maps in $S$ are acyclic, then statement (4) of Theorem \ref{thm:functioncount} implies that for each $Q_i$ there are positive constants $c_{S,Q_i}$ such that 
\begin{equation}\label{step2}
\#\{f\in M_S\,:\, h(f(Q_i))\leq X\}\sim c_{S,Q_i}X^{\rho}.
\end{equation} 
Here, we use that the $Q_i$ are not preperiodic. In particular, combining \eqref{step1} and \eqref{step2} yields the desired conclusion.  Likewise, the general case follows from \eqref{step1} and statement (3) of Theorem \ref{thm:functioncount}.

\end{proof}
In particular, if we fix integers $d_1,\dots d_r$ all at least four, then the bound for orbits in Theorem \ref{thm:rational} holds on a Zariski open subset of $\Rat_{d_1}\times\dots\times\Rat_{d_r}$.    
\begin{proof}[(Proof of Theorem \ref{thm:generic})] For each $d_i$, the proof of \cite[Theorem 1.4]{Pakovich} implies that there is a proper open subset $U_i\subset\Rat_{d_i}$ such that the curve
\begin{equation}\label{curvei}
C_\phi:=\Big\{(x,y)\,:\, \frac{\phi(x)-\phi(y)}{x-y}=0\Big\} 
\end{equation} 
is irreducible and of genus at least $2$ for all $\phi\in U_\phi$. Likewise, the proof of \cite[Theorem 1.2]{Pakovich} implies that for each pair $(d_i, d_j)$ there are proper open subsets $V_{ij}\subset \Rat_{d_i}$ and $W_{ij}\subset\Rat_{d_j}$ such that the curve
\begin{equation}\label{curveij}
C_{\phi,\psi}:=\Big\{(x,y)\,:\, \phi(x)=\phi(y)\Big\}
\end{equation} 
is irreducible and of genus at least $2$ for all $\phi\in V_{ij}$ and all $\psi\in W_{ij}$. Now define the open set 
\[
\scalemath{.95}{
U:=O_1\times\dots\times O_r\subset\Rat_{d_1}\times\dots\times\Rat_{d_r},\,\;\;\text{where}\;\;\, O_i=\bigcap_{j=1}^r(U_i\cap V_{ij}\cap W_{ji}).}
\] 
In particular, if $(\phi_1,\dots,\phi_r)\in U$, then $C_{\phi_i}$ and $C_{\phi_i,\phi_j}$ are irreducible and genus at least $2$ for all $i$ and $j$. Now let $S=\{\phi_1,\dots, \phi_r\}$. Then the conclusions of Proposition \ref{thm:fstops} hold for $S$ - the proof of that result uses only that $C_{\phi_i}$ and $C_{\phi_i,\phi_j}$ have finitely many rational points, which follows from Faltings theorem. In particular, as in the proof of Theorem \ref{thm:rational} above, the desired asymptotics now follow directly from combining statements (3) or (4) of Theorem \ref{thm:functioncount} with Proposition \ref{thm:fstops}.          
\end{proof} 
Although Theorems \ref{thm:rational} and \ref{thm:generic} do not apply directly to polynomials, similar statements can be made in this case assuming the disjointness and simplicity of \emph{affine} critical values and that the degrees of the polynomials in $S$ are at least $5$; see, for instance, \cite[\S3]{poly:functional}. However, such conditions are stronger than those in the following result; compare to \cite[Theorem 1.6]{Wade:MathZ}. 
\begin{thm}\label{thm:poly}
    Let $S=\{\phi_1,\dots,\phi_r\}$ be a set of polynomials defined over a number field $K$ with distinct and acyclic degrees all at least $2$. Moreover, let $\rho$ be the real number satisfying $\sum_{\phi \in S} \deg(\phi)^{-\rho}=1$ and assume that the following conditions hold: \vspace{.15cm} 
    \begin{enumerate}
        \item Each $\phi_i \in S$ is not of the form $R \circ E \circ L$ for some polynomial $R \in \overline{\Q}[x]$, some cyclic or Chebyshev polynomial $E$ with $\deg(E)\geq2$, and some linear polynomial $L \in \overline{\Q}[x]$; \vspace{.15cm} 
        \item For all $i \neq j$, we have that $\phi_j \neq \phi_i \circ F$ for any $F \in \overline{\Q}[x]$. \vspace{.15cm} 
    \end{enumerate}
    Then $M_S$ is a free semigroup and for all points $P\in \P^1(K)$ such that $\Orb_S(P)$ is infinite we have that 
    $$
    \# \{Q \in \Orb_S(P) : h(P) \leq X\} \sim cX^\rho \vspace{.1cm} 
    $$
    for some positive constant $c$ depending on $S$ and $P$. 
\end{thm}
\begin{proof}
     The semigroup $M_S$ is free by \cite[Theorem 1.6]{Wade:MathZ}. For the rest, take $\mathcal{O} \subset K$ to be a ring of $\mathcal{S}$-integers in $K$ containing $P$ and all the coefficients of polynomials in $S$. Then \cite[Proposition 4.5]{Wade:MathZ} implies that the curves $C_{\phi}$ and $C_{\phi,\psi}$ for all $\phi,\psi\in S$ defined in \eqref{curvei} and \eqref{curveij} have finitely many $\mathcal{O}$-points; that is, the set  
     \[\Sigma:=\{(x,y)\in \mathcal{O}\times\mathcal{O}\,: (x,y)\in C_{\phi}\;\text{or}\; (x,y)\in C_{\phi,\psi}\;\text{for some $\phi,\psi\in S$}\} \] 
     is finite. From here, an identical proof to that of statement (3) of Proposition \ref{prop:can} implies that $M_S$ has the cancellation property; specifically, if 
\[B_0 = \max_{(x,y) \in \Sigma}\big\{\max\{h(x),h(y)\}\big\},\] 
if $B_1$ is the constant from Lemma \ref{lem:hts}, and if $B = \max\{B_0, B_1\}$, then $B=B(P)$ satisfies the conditions of Definition \ref{def:cancel}. In particular, the desired asymptotics now follow directly from Proposition \ref{prop:sum} and Theorem \ref{thm:functioncount}. 
     \end{proof}


\bibliographystyle{plain}
\bibliography{HieghtCounting.bib}

\end{document}